\documentclass[a4paper, 11 pt, twoside]{amsart}
\usepackage{srollens-en}[2014/04/11]
\usepackage[left = 3.5cm, right = 3.5cm, headsep = 6mm,
footskip = 10mm, top = 35mm, bottom = 35mm, footnotesep=5mm, headheight =
2cm]{geometry}

\usepackage[usenames, dvipsnames]{xcolor}
\usepackage{tikz-cd}

\usepackage{booktabs, multirow}
\usepackage{enumitem}

\usepackage{pdflscape}

\usepackage[unicode,bookmarks, pdftex, backref = page]{hyperref}
\hypersetup{colorlinks=true,citecolor=NavyBlue,linkcolor=NavyBlue,urlcolor=Orange, pdfpagemode=UseNone, breaklinks=true}

\renewcommand{\tilde}{\widetilde}
\newcommand{\I}{\mathrm{i}}

\DeclareMathOperator{\Sym}{Sym}
\DeclareMathOperator{\Diff}{Diff}
\DeclareMathOperator{\Res}{Res}

\title{Canonical rings of Gorenstein stable Godeaux surfaces}

\author{Marco Franciosi}
\address{Marco Franciosi\\Dipartimento di Matematica\\Universit\`a di Pisa \\Largo B. Pontecorvo 5\\I-56127  Pisa\\Italy}
\email{marco.franciosi@unipi.it}

\author{S\"onke Rollenske}
\address{S\"onke Rollenske\\FB 12/Mathematik und Informatik\\
Philipps-Universit\"at Marburg\\
Hans-Meerwein-Str. 6\\
35032 Marburg\\
Germany}
\email{rollenske@mathematik.uni-marburg.de}

\begin{document}
\begin{abstract}
Extending the description of canonical rings from \cite{reid78} we show that every Gorenstein stable Godeaux surface with torsion of order at least $3$ is smoothable.
\end{abstract}

\maketitle

\setcounter{tocdepth}{1}
\tableofcontents
\section{Introduction}
Surfaces of general type with the smallest possible invariants, namely $K_X^2=1$ and $p_g(X) = q(X) =0$ are called (numerical) Godeaux surfaces, in  honour of L. Godeaux who constructed the first such examples. 

It is well known that their algebraic fundamental group, or equivalently, the torsion subgroup  $T(X)\subset \Pic(X)$, is cyclic of order at most $5$. After his seminal work \cite{reid78}, Reid was lead to the following.
\begin{custom}[Conjecture (Reid)]
The fundamental group of a Godeaux surface is cyclic of order $d\leq 5$ and for every $d$ the moduli space of Godeaux surfaces with fundamental group $\IZ/d$ is irreducible and (close to) rational.

In particular, the  Gieseker moduli space $\gothM_{1,1}$ has exactly five irreducible components.
\end{custom}
Indeed, he had shown this to be true as soon as $|T(X)|\geq 3$ in \cite{reid78}. While quite a few  examples of Godeaux surfaces with torsion $\IZ/2$ or trivial have been constructed since then   (see e.\,g.\ \cite{Barlow84, Barlow85, catanese-debarre89, Inoue94, Werner94, Dol-Wer99, Dol-Wer01, coughlan16}) 
the conjecture is still open. For instance, it is still possible that there is a Godeaux surface with trivial algebraic fundamental group, whose fundamental group is infinite.

The purpose of this article is to extend the results from \cite{reid78} to Gorenstein stable Godeaux surfaces.

\begin{custom}[Theorem A]
 The moduli space of Gorenstein stable Godeaux surfaces with order of the torsion subgroup $|T(X)|= 3, 4, 5 $ is irreducible. 
 
 In particular, every Gorenstein stable Godeaux surface with $|T(X)|= 3,4,5$ is smoothable. 
\end{custom}

Stable surfaces are the the analogue of stable curves in dimension two: they are the surfaces that appear in a modular compactification of the Gieseker moduli space of surfaces of general type 
(see  \cite{alexeev06},  \cite{kollar12}, \cite{KollarModuli}). Indeed, this was one of the motivations for the introduction of this class of singularities by Koll\'ar and Shepherd-Barron in \cite{ksb88}. 
 However, this compactification, called the moduli space of stable surfaces, can have many additional irreducible or connected components containing non-smoothable surfaces. This happens also for Gorenstein stable Godeaux surfaces \cite{rollenske16}.

The method  of \cite{reid78}, studying the canonical ring of a maximal finite cover (via restriction to a canonical curve),  generalises quite well to Gorenstein stable surfaces.  
The only obstacle is that we cannot assume the general member of a base-point free linear system to be smooth or irreducible, that is, we have to deal with slightly worse curves throughout. As an example of the flavor of arguments that need to be reconsidered carefully: a non-hyperelliptic smooth curve of genus $4$ embeds canonically in $\IP^3$ as the complete intersection of a quadric and a cubic, while the image of the canonical map of the union of an elliptic curve and a non-hyperelliptic curve of genus $3$ meeting in a node is the union of a plane quartic and a point not in that plane.

In a companion paper \cite{FPR16b} we will study more in detail examples of Gorenstein stable Godeaux surfaces and in particular prove that $T(X)$ {is a cyclic group}	 of order at most $5$.
This article is  part of series of papers, mostly joint with Rita Pardini, exploring Gorenstein stable surfaces with $K_X^2=1$. These are tamed by the fact that $1$ is a very small number but every now and then some unexpected phenomena occur.

\subsection*{Acknowledgements}
The results in this paper are part of our exploration of Gorenstein stable surfaces with $K_X^2 = 1$ carried out jointly with Rita Pardini. We would like to thank her for this enjoyable collaboration.

The second author enjoyed several discussions with Stephen Coughlan and Roberto Pignatelli about canonical rings in general and Godeaux surfaces in particular.

The first author is grateful for support by the PRIN project  2010S47ARA$\_$011  ``Geometria delle Variet\`a Algebriche'' of italian MIUR. 

The second author is grateful for support by the DFG via the Emmy Noether programme and partially SFB 701.

\section{Notations and conventions}\label{sec:intro}
We work exclusively with schemes of finite type over the complex numbers.
\begin{itemize}
\item A curve is a projective scheme of pure dimension one but not necessarily reduced or irreducible or  connected.
\item A surface is a reduced, connected projective scheme  of pure dimension two but not necessarily irreducible.
\item We will not distinguish between Cartier divisors and invertible sheaves. 
\item All schemes we consider will be Cohen-Macaulay and thus admit a dualising sheaf $\omega_X$. We call $X$ Gorenstein if $\omega_X$ is an invertible sheaf. 
A canonical divisor is a Weil-divisor $K_X$ whose support does not contain any component of the non-normal locus and such that $ \ko_X(K_X)\isom \omega_X $.
\item Given a an invertible sheaf $L\in \Pic(X)$,  one defines the ring of sections 
\[R(X, L)= \bigoplus_{m\geq 0}H^0(mL);\] for $L=K_X$, we have the \emph{ canonical ring}  $R(K_X):=R(X,K_X)$.
\end{itemize}

\subsection{stable surfaces}
Since we work on stable surfaces we feel compelled to give a definition. However, the precise nature of the singularities of stable surfaces will only indirectly play a role in the sequel, namely when we use Riemann--Roch and Kodaira vanishing to compute the dimension of some spaces of sections. Our main reference  is \cite[Sect.~5.1--5.3]{KollarSMMP}.

Let $X$ be a demi-normal surface, that is,  $X$ satisfies $S_2$ and  at each point of codimension one $X$ is either regular or has an ordinary double point.
We denote by  $\pi\colon \bar X \to X$ the normalisation of $X$. The conductor ideal
$ \shom_{\ko_X}(\pi_*\ko_{\bar X}, \ko_X)$
is an ideal sheaf in both $\ko_X$ and $\ko_{\bar X} $ and as such defines subschemes
$D\subset X \text{ and } \bar D\subset \bar X,$
both reduced and of pure codimension 1; we often refer to $D$ as the non-normal locus of $X$. 

The  demi-normal surface $X$ is said to have \emph{semi-log-canonical (slc)}  singularities if it satisfies the following conditions: 
\begin{enumerate}
 \item The canonical divisor $K_X$ is $\IQ$-Cartier.
\item The pair $(\bar X, \bar D)$ has log-canonical (lc) singularities. 
\end{enumerate}
It  is called a stable  surface 
 if in addition $K_X$ is ample. We define the geometric genus of $X$ to be $ p_g(X) = h^0(X, \omega_X) = h^2(X, \ko_X)$ and the irregularity as $q(X) = h^1(X, \omega_X) = h^1(X, \ko_X)$, so that we have $\chi(X) := \chi(\ko_X) =  1-q(X) +p_g(X)$.
A Gorenstein stable surface is a stable surface such that $K_X$ is a Cartier divisor. 

We will discuss surfaces in the following hierarchy of moduli spaces of surfaces with fixed invariants $a=K_X^2$ and $b=\chi(X)$.
\[
 \begin{tikzcd}[row sep = small]
\gothM_{a,b} \dar[hookrightarrow]&=&\text{Gieseker moduli space of surfaces of general type}\\
\overline\gothM_{a,b}^{(G)} \dar[hookrightarrow]&=&\text{moduli space of Gorenstein stable surfaces}\\    
\overline\gothM_{a,b} &=&\text{moduli space of stable surfaces}
 \end{tikzcd}
\]

For the time being there is no  self-contained reference for the existence of the moduli space of stable surfaces with fixed invariants as a projective scheme, and we will not use this explicitly. A major obstacle in the construction is that in the  definition of the moduli functor one needs additional conditions beyond flatness to guarantee that invariants are constant in a family. For Gorenstein surfaces these problems do not play a role; we refer to \cite{kollar12} and the forthcoming book \cite{KollarModuli} for details.

% One can decompose these moduli spaces according to geometric genus and irregularity which are constant in a family of stable varieties by \ref{Kovacs/Kollar??}.

\section{Paracanonical curves on Gorenstein stable Godeaux surfaces}

In analogy with the smooth case, we define a stable (numerical)  Godeaux surface as a Gorenstein stable surface $X$ with $K^2_X=1$ and $p_g(X)=q(X)=0$. 
This is the same as asking for $K^2_X=\chi(X)=1$ by \cite[Prop. 4.2]{FPR15b}.

For a stable Godeaux surface the group $T(X)$ of torsion {invertible sheaves} 
coincides with $\Pic^0(X)$ and is a finite group. 

For a smooth Godeaux surface it is classically known that $T(X)$ is cyclic of order at most $5$ and in   \cite{FPR16b} we show that no surprises occur for stable Godeaux surfaces. 
\begin{prop}[{\cite[Cor.~4.2]{FPR16b}}]\label{prop: pi1 cyclic}
Let $X$ be a Gorenstein stable Godeaux surface. Then $\pi_1^{\text{alg}}(X)$, and hence $T(X)$, is a cyclic group  of order at most $5$. 
\end{prop}

We will now analyse paracanonical curves on a Gorenstein stable Godeaux surface $X$, that is curves numerically equivalent to $K_X$.

\begin{lem}\label{lem: paracanonical 1}
 Let $C_i$ be distinct Cartier divisors which are numerically equivalent to $K_X$. Then
\begin{enumerate}
\item $C_i$ is an  irreducible Gorenstein curve of arithmetic genus 2,
 \item $C_1$ and $C_2$ intersect with multiplicity 1 in a single point $P(C_1, C_2)$,
\item if $C_1-C_2$ is not linearly equivalent to $0$, then $P(C_1, C_3) \neq P(C_2, C_3)$.
\end{enumerate}
\end{lem}
\begin{proof}
 By assumption we have $K_XC_i=K_X^2=1$, thus $C_i$ is irreducible as $K_X$ is an ample Cartier divisor and Gorenstein. Since by adjunction $K_{C_i} = (K_X+C_i)|_{C_i}$, the first and the second item follow. For the last item consider
\[ 0\to \ko_X(C_1-C_2-C_3)\to \ko_X(C_1-C_2)\to \ko_{C_3}(P(C_1, C_3) - P(C_2, C_3))\to 0.\]
In the associated long exact sequence in cohomology $H^0(\ko_X(C_1-C_2))=0$ by assumption and $H^1( \ko_X(C_1-C_2-C_3)=0$ by Kodaira vanishing (see e.\,g.\ \cite[Cor.~19]{liu-rollenske14}). Thus $H^0(\ko_{C_3}(P(C_1, C_3) - P(C_2, C_3)))=0$ and the last claim follows.
\end{proof}

 The existence of a paracanonical curve is ensured as soon as $T(X)\neq 0$.
\begin{lem}\label{lem: paracanonical 2}
Let $L_1$ be a torsion invertible sheaf  on $X$. Then
\begin{enumerate}
 \item $h^0(m(K_X+L_1)) = 1+\binom m 2$ for $m\geq 2$.
\item If $L_1$ is non-trivial then $h^0(K_X+L_1) \geq  1$ and equality holds if there exists a different non-trivial torsion invertible sheaf  $L_2$ on $X$.
\end{enumerate}
\end{lem}
\begin{proof}
 The first item follows directly from the Riemann-Roch formula \cite[Thm. 3.1]{liu-rollenske16} and Kodaira vanishing \cite[Cor.~19]{liu-rollenske14}. 
In the second case $h^2(K_X+L_i)=h^0(L_i)=0$ since $L_i$ is non-trivial thus $h^0(K_X+L_1)=h^1(K_X+L_1) +1\geq 1$. To show equality let $C_2$ be a section of $H^0(K_X+L_2)$, which exists by Riemann-Roch and is irreducible of arithmetic genus 2 by Lemma \ref{lem: paracanonical 1}. Thus the invertible sheaf  of degree 1  $(K_X+L_1)\restr C_2 $ has at most 1 section and the restriction sequence
\[0\to H^0(L_1-L_2)\to H^0(K_X+L_1)\to H^0((K_X+L_1)\restr {C'})\to \dots\]
gives $h^0(K_X+L_1)\leq1$. This concludes the proof.
\end{proof}

For later reference we compute the dimension of some cohomology spaces.

\begin{lem}\label{lem: dimension on paracanonical curves}
Let $X$ be a Gorenstein stable Godeaux surface with $T(X)$ cyclic of order $d\geq 3$,  $L$ be a generator for $T(X)$. Let $C\in |K_X+L|$ be  a paracanonical curve in $X$ and $M = K_X|_C$. Then $2M+L = K_C$ and 
\begin{align*}
 &h^0(C, M+iL) = \begin{cases}
                          0 &i=0,1\\ 1& i = 2, \dots ,d-1
                         \end{cases}\\
& h^0(C, 2M+iL) = \begin{cases}
                          1 &i\neq 1\\ 2& i = 1
                         \end{cases}\\
& h^0(C, mM+iL) =m-1 \quad m\geq3
\end{align*}
\end{lem}
\begin{proof}
 Since $C\in |K_X+L|$ is a Gorenstein curve of genus 2 everything, follows from the restriction sequence, the invariants of $X$, adjunction and Riemann--Roch on $C$.
\end{proof}

We will be studying graded section rings on curves. To control multiplication maps we need Castelnuovos base-point-free pencil trick and its variants.

 \begin{prop} \label{prop:van}  Let $C$ be a  reduced and connected  Gorenstein curve and let $F$, $H$ be  invertible sheaves on $C$. 
 Assume   $W \subseteq H^{0}(C, F) $ 
is a subspace of $\dim = r+1$ which defines  
	a base point free system. Then 
 \begin{enumerate}
	 \item  If   
$H^1(C,  H\otimes F^{-1})=0$,   then the multiplication map
\[ W \otimes H^0(C,H) \rightarrow H^0(C,F\otimes H) \]
is surjective.
   \item If $r=1$, i.\,e., $W$ is a base point free pencil,  then 
   \[\ker \{W \otimes H^0(C,H) \rightarrow H^0(C,F\otimes H) \} \cong H^0(C, H \otimes F^{-1}).\]
 \end{enumerate}
\end{prop}

\begin{proof} The first part follows from \cite[Prop. 25]{franciosi13}. 
 
In the latter case  we can repeat verbatim  the proof of the base-point-free pencil trick given  in \cite[chap. III , \S 3, pag. 126]{ACGH}.  
\end{proof}

\begin{rem}
In the following sections we study  canonical rings of surfaces by restriction to canonical curves $C$. One step in the analysis is usually to show that the Gorenstein curve $C$ behaves like a smooth curve of the same genus. The relevant notion in this context is numerical connectedness, which was defined in \cite{CFHR}, and by \cite{franciosi-tenni14} everything would follow if $C$ is numerically $3$-connected.
Checking $3$-connectedness turns out to be at least as intricate as a direct proof of the properties we need in each  cases.  Thus we decided to avoid this extra layer of complexity.
\end{rem}

\section{Canonical ring of Gorenstein stable  Godeaux surfaces with large torsion}

In this section we adapt the algebraic treatment of numerical Godeaux surfaces with sufficiently large torsion by Miles Reid to the Gorenstein case. Many arguments carry over unchanged from \cite{reid78}. The idea is that, if $\pi_1^{\alg}$ is sufficiently big, then the universal cover is simple to describe algebraically. The main issue is that due to the presence of singularities we have to work on considerably more singular curves.

\subsection{The case $|T(X)|=5$} 
\begin{thm}
Consider the action of $G=\IZ/5$ on $S = \IC[x_1, \dots, x_4]$ given by $x_i\mapsto \xi^ix_i$ where $\xi$ is a primitive fifth root of unity. 

 Let $X$ be a stable Gorenstein surface with $\pi_1^\alg(X) = \IZ/5$ and $Y$ the (algebraic) universal cover. Then $Y$ is a quintic surface in $\IP^3$ defined by a $G$-invariant  quintic polynomial $q\in S$ and the canonical ring of $X$ is
\[ R(X, K_X) = \left(S/(q)\right)^G.\]
\end{thm}
\begin{proof}
 The proof in  \cite{reid78} works without modification. 
\end{proof}
\begin{rem}\label{rem: X15 as a quintic}
 Consider a quintic given by the orbit of a plane not meeting the fixed points of the above action, for example given by 
\begin{multline*} q = (x_1+x_2+x_3+x_4)(\xi x_1+\xi^2x_2+\xi^3x_3+\xi^4x_4)(\xi^2 x_1+\xi^4x_2+\xi x_3+\xi^3x_4)\\ \cdot
(\xi^3 x_1+\xi x_2+\xi^4 x_3+\xi^2x_4)(\xi^4 x_1+\xi^3 x_2+\xi^2 x_3+\xi x_4).
 \end{multline*}

The surface $Y = \{q=0\}$ is a normal crossing divisor in $\IP^3$, in particular stable. It has exactly 10 triple points, thus the quotient by the free action is a stable surface with normalisation $\IP^2$ and two triple points. This is the surface $X_{1,5}$ described in \cite[Sect.~4.2]{FPR15b}.
\end{rem}

\subsection{The case $|T(X)|=4$} 
By Proposition \ref{prop: pi1 cyclic} in this case $T(X)\isom \IZ/4$. This case can be described explicitly.
\begin{thm}
Consider the polynomial ring $S =\IC[x_1, x_2, x_3, y_1, y_3]$ with $\deg x_j =1$ and $\deg y_j = 2$. An action of $G=\IZ/4$ on $S$ is defined by $x_j\mapsto \I^j x_j$ and $y_j \mapsto \I^jy_j$. 
  Let $X$ be a stable Gorenstein surface with $T(X) = \IZ/4$ and $f\colon Y\to X$ be the associated cover. Then there are two $G$-invariant polynomials $q_1, q_2\in S$ of weighted degree 4 such that 
\[ R(Y, K_Y) = S/(q_1, q_2)\text{ and } R(X, K_X)=  R(Y, K_Y)^G.\]

In other words, $X$ is the free quotient by $G$ of a weighted complete intersection of bidegree $(4,4)$ in $\IP(1,1,1,2,2)$.
\end{thm}
\begin{proof}
 Fix a generator $L$ of $T(X)$ and let $C$ be the paracanonical  curve corresponding to a generator  $x_1\in H^0(X, K_X+L)$.
  We set $M = K_X|_C$. The preimage $\tilde C = \inverse f(C) $ is a canonical curve in $Y$ and has arithmetic genus $5$. 
  
  Let  $\tilde M = K_Y|_{\tilde C} = f^*M$.
 We now compute the restriction of the canonical ring of $Y$ to $\tilde C$, which via $f_*$ can also be interpreted as a $\IZ\times \IZ/4$ graded ring on $C$, that is,
 \[R = R(Y, K_Y)|_{\tilde C} = R(\tilde C, \tilde M ) = \bigoplus_{m} H^0(m\tilde M) = \bigoplus_m\bigoplus_{i=0}^3H^0(C, mM+iL).\]
 This decomposition is the weight-space decomposition with respect to the $\IZ/4$-action on $\tilde C$. Note that we computed all the relevant dimensions in Lemma \ref{lem: dimension on paracanonical curves}.
 
 Let $x_2$ be a generator of $H^0(M+2L)$ and $x_3$ be a generator of $H^0(M+3L)$. Then $H^0(\tilde M)=\langle x_2, x_3\rangle$ defines a base-point-free pencil on $\tilde C$ by Lemma \ref{lem: paracanonical 1}. We now choose  $y_1\in H^0(2 M+L)$ and $y_3\in H^0(2M+3L)$ such that they span a complement of 
 $ \im \{ S^2 H^0(\tilde M) \to H^0(2\tilde M)\}$.

 We claim that $x_2, x_3, y_1, y_3$ generate $R$. 
Indeed for every $k \geq 2$ consider the multiplication map 
\[ \mu_{1,k}\colon  H^0(\tilde M)\tensor H^0(k\tilde M)\to H^0((k+1)\tilde M).  \]
 It is surjective for $k=2$ since its Kernel is isomorphic to  $H^0(\tilde M)$   by (ii) of  Proposition \ref{prop:van}, whilst  $h^0(2\tilde M) =g(\tilde C)=5$ and    $h^0( 3\tilde M) =8$  by Lemma
  \ref{lem: dimension on paracanonical curves}.
For $k\geq 4$ it is surjective  by   Prop. \ref{prop:van} \refenum{i} as  $H^1( (k-1)\tilde M) =0$. 
For $k=3$ the image has codimension $1$, and it remains to show that there is no additional generator in degree $4$. 
Considering a second multiplication map $\mu_{2,2}\colon  H^0(2\tilde M)\tensor H^0(2\tilde M)\to H^0(4\tilde M)$ our claim follows if  
  $$\im (\mu_{2,2}) + \im (\mu_{1,3}) = H^0(4\tilde M)$$
as vector spaces. 

To prove this, we first decompose image and target of $\mu_{1,3}$ into weight spaces to identify the culprit for non-surjectivity: it is the sequence
{\small 
\[
\begin{tikzcd}[column sep = small]
0\rar&
H^0(K_C) \arrow{rr}{\begin{pmatrix} x_2\\-x_3\end{pmatrix}} &&
\begin{matrix}  H^0(M+3L)\tensor H^0(3M+3L) \\ \oplus   H^0(M+2L)\tensor H^0(3M) \end{matrix}\rar& H^0(2K_C)\rar& H^1(K_C)\rar &0
\end{tikzcd}
\]
}
where we used $K_C = 2M+L$. 

Therefore it suffices to show that the image of $\mu_{2,2}$ contains  $H^0(2K_C)$.
But since $C$ is an integral Gorenstein curve of genus $2$ the canonical linear system is also a base-point free pencil and the long exact sequence for 
\[ 0\to \ko_C \to K_C\tensor H^0(K_C) \to 2K_C\to 0\]
shows that the multiplication map $ H^0(K_C)\tensor H^0(K_C)\to H^0(2K_C)$ is surjective. 

Therefore $R$ is generated in degree at most two and hence by $x_2, x_3, y_1, y_3$. 
Counting dimensions one checks that the the kernel of $\IC[x_2, x_3, y_1, y_3]\onto R$ is generated by two relations in weighted  degree $4$ (of weight $0$ and $2$).

Thus,  $R(Y, K_Y)|_{\tilde C} = R(\tilde C, \tilde M )$ realises $ \tilde C$ as a complete intersection of degree $(4,4)$ in $\IP(1,1,2,2)$ which implies that the canonical ring of $Y$ realises $Y$  as a complete intersection of degree $(4,4)$ in $\IP(1,1,1,2,2)$.

The rest of the statements carries over verbatim from \cite[\S2]{reid78}.
\end{proof}

\subsection{The case $|T(X)|=3$}
We now consider the following situation. Let $X$ be a Gorenstein stable numerical Godeaux surface with $T(X)=\IZ/3$, generated by $L$, and   let $C$ be the paracanonical  curve corresponding to a generator  $x_1\in H^0(X, K_X+L)$. Let $f\colon Y\to X$ be the associated \'etale cover of $X$. As in the previous case we set
 \begin{gather*}
 M = K_X\restr C, \, L = L\restr C,\\
\tilde C = \inverse f (C)\in |K_Y,|\\
\tilde M = f^*M = K_Y\restr {\tilde C}.
\end{gather*}
We will analyse the canonical ring of $Y$ via restriction to $\tilde C$. 

We now introduce some further notation. Let $R=R(\tilde C, \tilde M)$ be the graded ring of sections associated to $\tilde M$. Push-forward to $C$ induces an additional $\IZ/3$ grading on this ring. Writing 
\[ R_i^m = H^0(C, mM+iL) ,   \ \ R^m = \left(R^m_0\oplus R^m_1\oplus R^m_2\right)\]
we have $R = \bigoplus R^m$, where 
we consider the $R_i^m$ weight-spaces for an appropriately normalised $\IZ/3$-action.

We can choose  elements $x_2, y_0, y_1, y_2\in R$ such that
\begin{equation}\label{eq: choice of basis}
\begin{split}
&R^1_2 = H^0(\ko_C(P))= \langle x_2\rangle,\\
&R^2_0 = \langle y_0\rangle, R^2_1 = \langle x_2^2, y_1\rangle, R^2_2 = \langle y_2\rangle,\\
&(y_0:y_1:y_2)(\tilde P) = (1:1:1).
\end{split}
\end{equation}
Note that we have made use of the fact that $C$ is reduced to ensure that $x_2^2\neq 0$. Since $C$ is integral of genus $2$ the sections of $R_1^2 = H^0(K_C)$ define a base-point-free pencil on $C$ and we get a polynomial subring $\Sym^*R^2_1 \subset R$. Thus $y_1$ cannot vanish at $P$ and also $y_0$ and $y_2$ cannot vanish at $P$ because $R_1^1=R^1_0 = 0$; thus  we can scale the sections to satisfy the last condition.

Denoting the canonical map of $\tilde C$ by 
\[ \phi=\phi_{K_{\tilde C}}\colon \tilde C \dashrightarrow \IP(R^2)^\vee\isom\IP^3\]
the action of $\IZ/3$ on the canonical model is induced by an action on $\IP^3$. 

\begin{prop}\label{prop: canonical ring C tilde}
 Let $R(\tilde C, K_{\tilde C}) = R^{[2]} = \bigoplus_{m\geq0} R^{2m} $ be the subring of even elements of $R$, which coincides with  the canonical ring of $\tilde C$ (with degrees multiplied by $2$). Choosing appropriate elements  satisfying the conditions \eqref{eq: choice of basis} we have
 \[R^{[2]}\isom \IC[x_2^2, y_0, y_1, y_2]/(f_2, h_0)\]
 where 
\[
\begin{split}
 f_2 &= y_0y_2-y_1^2+x_2^4,\\
h_0&= y_0^3-2y_0y_1y_2+y_2^3+x_2^2F_2,\\
& \text {with } F_2 = \alpha x_2^4+\beta x_2^2y_1+\gamma y_0y_2, \quad \alpha, \beta, \gamma\in \IC.
\end{split}
\]
In particular, the canonical map embeds $\tilde C$ as a complete intersection of a cubic and a quadric in $\IP^3$.
\end{prop}
\begin{proof}
 We first show that $ R^{[2]} $ is indeed generated by the elements of degree $2$. First, let us consider $R^4= H^0(4M) \oplus H^0(4M+L)\oplus H^0(4M+2L) $.
We have 
\begin{align*}
\langle x_2^2 y_2,  y_1y_2, y_0^2  \rangle  \subset  H^0(4M) \\ 
\langle x_2^2 y_0,  y_0y_1, y_2^2  \rangle  \subset   H^0(4M + L) \\
\langle x_2^4,  y_1^2, x_2^2y_1  \rangle  \subset  H^0(4M+2L) 
\end{align*} 
and we claim that in each case there is no relation between the elements on the left hand side. This is clear for $H^0(4M+2L)= \Sym^2 H^0(2M+L)$.  Since $H^0(2M+L)$ defines a base-point-free pencil every relation in $\langle x_2^2 y_2,  y_1y_2, y_0^2  \rangle $ can be written as $y_0^2 = y_1(\alpha x_2^2+\beta y_1)$. But $\mathrm{div}\, \frac {y_0}{y_1} \sim 2L$ is non-trivial, so the left hand side and the right hand side can never have the same zeros, a contradiction. The argument in the other case is the same.
Then counting dimensions shows that we have equality in all three cases, that is,  
$ R^2 \tensor R^2  \onto  R^4$.

Now let us consider $R^6= H^0(6M) \oplus H^0(6M+L)\oplus H^0(6M+2L) $. 

From the base-point-free pencil trick (Proposition \ref{prop:van}) we get 
\begin{align*}
H^0(2M+L)\otimes H^0(4M) \onto H^0(6M+L) \\
H^0(2M+L)\otimes H^0(4M+L) \onto H^0(6M+2L)
\end{align*} 
since $H^1(2M-L)= 0$ and $H^1(2M)=0$ because  $2M+L \cong K_C$. 
To complete the remaining summand observe that 
\[\ker \{H^0(2M+L)\otimes H^0(4M+2L) \to H^0(6M)\}  \cong  H^0(C, 2M+L)=H^0(C,K_C).\] 
Hence the cokernel is one-dimensional  since $h^0(6M)=5$. By an argument similar to the above we see that $(y_0)^3$ is not in the image of this multiplication map and thus  
\[ H^0(6M)= H^0(2M)\cdot H^0(4M)+  H^0(2M+L)\cdot H^0(4M+2L)\]
is generated by products of  elements  in $R^2$.

For higher degrees Proposition \ref{prop:van} implies the surjectivity of 
\[ H^0(2M+L)\otimes H^0((2hM + iL ) \onto H^0((2h+2)M +(i+1)L)  \]
for every $h\geq 3$ and every $i \in \IZ/3$  because  $H^1((2h-2)M+(i-1)L)=0$  for degree reasons.

Therefore $ R^2 \otimes R^{2h} \onto R^{2h+2}$ for every $h\geq 3$,
which shows that $R^{[2]}$ is generated in degree $2$. In particular,  the surjection $R^{2}\isom \IC[x_2^2, y_0, y_1, y_2]\to R^{[2]}$ 
 defines the canonical embedding $\phi = \phi_{K_{\tilde C}}\colon \tilde C \to \IP^3$,  $\phi (\tilde C)$  a reduced curve of degree 6 and $p_a( \tilde C) =4$, and $\phi(\tilde P) =(0:1:1:1)$ because of \eqref{eq: choice of basis}. 

Counting dimensions we see that $\phi(\tilde C)$ is contained in a unique quadric. Looking more closely at the basis for $R^4$ given above we see that there is a relation $f_2$ involving the missing element $y_0y_2\in R_2^4$. Noting that $f_2(\phi(\tilde P)) = f_2(0:1:1:1)=0$ we can arrange $f_2$ to be of the form above by completing the square with $y_1$ and rescaling $x_2$. This does not affect the choices made in \eqref{eq: choice of basis}. 

As in the classical case, $h^0( \IP^3, \mathcal{I}_{\phi(\tilde C)}(3))=5$, so we can find an irreducible cubic surface vanishing on  $\phi(\tilde C)$, which together with $f_2$ generates the ideal.
More precisely, dimension counting shows that the relation $h_0$ occurs among the monomials of degree $6$ and weight $0$.
Since the relation cannot be contained in $y_0\cdot H^0(4M)$ or $y_2\cdot H^0(4M+L)$ modulo $x_2$ it has to involve $y_0^3$, $y_2^3$ and $y_0y_2y_2$ with non-zero coefficients. Note that $y_0y_1y_2\equiv y_1^3 \mod x_2$ so we can eliminate $y_1^3$. We rescale the equation such that $h_0\equiv y_0^3-(a+1)y_0y_1y_2+ay_2^3 \mod x_2^2$. On $\tilde C$ the divisor of $x_2^2$ is $f^*2P$. Cutting the $\phi(\tilde C)$ with the plane $\{x_2^2=0\}$ thus gives three double points defined by $(f_2,h_0) \mod x_2^2$. Restricting the cubic equation to the quadric and computing the derivative we see that this happens if and only if $a=1$ and thus $h_2\mod x_2^2$ is of the claimed form.
\end{proof}

\begin{lem}\label{lem: S and z's}
Let $S= \bigoplus_{m\geq0} S^m$ be the subring of $R$ generated by $R^1$ and $R^2$. 
\begin{enumerate}
 \item With the choices made in Proposition \ref{prop: canonical ring C tilde} we have  
 \[R^{[2]}\subset S\isom \IC[x_0, y_0, y_1, y_2]/J\subset R,\]
 where $J=(f_2, h_0)$.
\item There exist elements $z_1\in R_1^3$ and $z_2\in R_2^3$ such that $R^3=S^3\oplus \langle z_1, z_2\rangle$ and the relations
\[ f_0=x_2z_1 +y_0^2-y_1y_2=0 \text{ and } f_1=x_2z_2 + y_0y_1-y_2^2=0 \]
hold  in $R^4$. Moreover, $z_1$ and $z_2$ do not vanish at $P$.
\end{enumerate}
\end{lem}
For convenient reference we collect generators and relations of $R$ in small degree in Table \ref{tab: can ring}.

\begin{landscape}
\newcommand{\notS}{}
\newcommand{\notbasis}[1]{#1}
\begin{table}[htb!]\caption{Canonical ring of the $\IZ/3$-cover of a Godeaux restricted to a canonical curve.}\label{tab: can ring}
\begin{center}
\small
\renewcommand{\arraystretch}{1.2}
 \begin{tabular}{ll ll ll}
 \toprule

degree &weights& monomials in $S$ & monomials not  in $S$ & relations in $J$ & relations not in $J$ \\
\midrule
  & $R^1_0$& \\
1 & $R^1_1$& \\
  & $R^1_2$& $x_2$ \\
\cmidrule{2-6}
  & $R^2_0$& $y_0$\\
2 & $R^2_1$& $x_2^2$, $y_1$\\
  & $R^2_2$& $y_2$ \\
\cmidrule{2-6}
  & $R^3_0$& $x_2^3$, $x_2y_1$\\
3 & $R^3_1$& $x_2y_2$& $z_1$\\
  & $R^3_2$& $x_2y_0$& $z_2$ \\
\cmidrule{2-6}
  & $R^4_0$&
  $x_2^2y_2$, $y_0^2$, $y_1y_2$& \notbasis{$x_2z_1$} &&
  $f_0 = x_2z_1+y_0^2-y_1y_2$\\

4 & $R^4_1$&
   $x_2^2y_0$, $y_0y_1$, $y_2^2$& \notbasis{$x_2z_2$}&
  &$f_1 = x_2z_2+y_0y_1-y_2^2$\\

  & $R^4_2$&
  $x_2^4$, $x_2^2y_1$, $y_0y_2$, \notbasis{$y_1^2$}&&
   $f_2 = x_2^4+y_0y_2-y_1^2$&\\
\cmidrule{2-6}
  & $R^5_0$&
  $x_2^3y_0$, $x_2y_0y_1$, $x_2y_2^2$& \notbasis{$x_2^2z_2$}, $y_1z_2$, \notbasis{$y_2z_1$} &
  & $x_2f_1$, $g_0\equiv y_1z_2-y_2z_1\mod x_2$\\

5 & $R^5_1$&
   $x_2^5$, $x_2^3y_1$, $ x_2y_0y_2$, \notbasis{$x_2y_1^2$} & $y_0z_1$, \notbasis{$y_2z_2$}&
  $x_2f_2$&$g_1\equiv y_0z_1-y_2z_2\mod x_2$\\

  & $R^5_2$&
  $x_2^3y_2$, $x_2y_0^2$, $x_2y_1y_2$ &\notbasis{$x_2^2z_1$}, $y_0z_2$, \notbasis{$y_1z_1$}&
  &$x_2f_0$, $g_2\equiv y_0z_2-y_1z_1\mod x_2$\\
\cmidrule{2-6}
\multirow{6}*{$6$}
   & \multirow{2}*{$R^6_0$}&
   $x_2^2\cdot R^4_2$,  & \dots &
  $x_2^2f_2$, $y_1f_2$,& $y_0f_0$, $y_2f_1$ \\
  && $y_0^3$, $y_0y_1y_2$, $y_1^3$, $y_2^3$ & &
  $h_0=  y_0^3-2y_0y_1y_2+y_2^3+x_2^4F_2
     %x_2^2\left(\alpha x_2^3+\beta x_2^2y_1+\gamma\right)
  $
  & $x_2g_1$,  $H_0$\\

  & \multirow{2}*{$R^6_1$}&
  $x_2^2\cdot R^4_0$ & \dots &
 \multirow{2}*{$y_2f_2$} & $x_2^2f_0$, $y_0f_1$, $y_1f_0$, \\
  && $y_0^2y_1$, $y_0y_2^2$, $y_1^2y_2$ & &
  &$x_2g_2$, $H_1$\\

  & \multirow{2}*{$R^6_2$}&
  $x_2^2\cdot R^4_1$,  &\dots &
  \multirow{2}*{$y_0f_2$}&$x_2^2f_1$, $y_1f_1$, $y_2f_0$,\\
  && $y_0^2y_2$, $y_0y_1^2$, $y_1y_2^2$  & & 
  &$x_2g_0$, $H_2$\\
\bottomrule
\end{tabular}
\end{center}
\end{table}

\end{landscape}

\begin{proof}
The first item is clear, because adding a square root of $x_2^2$ in $R^{[2]}$ cannot introduce new relations.

The existence of the $z_i$ follows from comparing $\dim R_i^3$ and $\dim x_2 R_{i+1}^2$, where we use again that $x_0$ cannot be a zero-divisor.  We prove the statement for $z_1$, the argument for $z_2$ is the same. 

 By the first item we have $x_2z_1 \in S_0^4=R_0^4$, that is,  $0\neq x_2z_1=ay^2_0+by_1y_2+cx_2^2y_2$ for some coefficients $a,b,c$.  Replacing $z_1$ by $z_1 - cx_2y_2$ we may assume that $c=0$.   Since the relation has to vanish at $\tilde P$ we also have $a = -b$ and both have to be non-zero. Rescaling $z_1$ we can choose $a = -1$ as claimed. 

For the final claim assume on the contrary that $z_1$ vanishes at $P$. Then $z_1/x_2\in H^0(2M+2L) = R^2_2$, which is spanned by  $y_2$. This is impossible, since $x_2y_2$ and $z_1$ are linearly independent.
\end{proof}

We will now show, that $R$ is generated in degree at most three and determine the relations in degree $5$. 
\begin{lem}\label{lem: relations g}
 The natural map $\bar R = \IC[x_2, y_0, y_1, y_2, z_1, z_2] \to R$ is surjective. Denoting its kernel by $I$ the relations in degree $5$ are
 \[ I^5 = \langle x_2f_1, g_0\rangle \oplus\langle x_2f_2, g_1\rangle\oplus \langle x_2f_0, g_2\rangle\]
 where the $g_i$ can be chosen to be
 \begin{align*}
 g_0 &= y_1z_2-y_2z_1+x_2^3y_0,\\
g_1& = y_0z_1-y_2z_2-x_2F_2,\\
g_2&=y_0z_1-y_2z_2+y_2x_2^3.
\end{align*}
\end{lem}
\begin{proof}
 Proposition \ref{prop:van} implies that on $\tilde C$ the multiplication map 
 \[H^0(m \tilde M )\tensor H^0(2\tilde M) \to H^0((m+2)\tilde M)\]
 is surjective for $m\geq 7$. Since we know that the subring of even elements is generated in degree $2$ and have dealt with degree $3$ in Lemma \ref{lem: S and z's} we only need to show that there is no new generator in degree $5$.
 
Let us discuss in detail the map 
\[ \bar R^5_0 = x_2\langle  x_2^2 y_0,  y_0y_1, y_2^2  \rangle+z_1\langle y_2\rangle+z_2\langle x_2^2, y_1\rangle \to R^5_0,\]
whose kernel contains the known relation $x_2f_1=x_2^2z_2 +x_2( y_0y_1-y_2^2)$ and at least one other relation $g_0$, which cannot be contained in  $x_2\langle  x_2^2 y_0,  y_0y_1, y_2^2  \rangle$. Thus $g_0$ has to contain at least one of the monomials $y_2z_1$ or $y_1z_2$. Since none of the $y_i, z_i$ vanish at $P$ while necessarily $g_0(P)=0$ we  can normalise to get 
\[ g_0 \equiv y_1z_2 - ay_2z_1 \mod x_2\qquad\text{ where } a = \frac{y_1z_2}{y_2z_1}( P) = \frac{z_2}{z_1}( P).\]
If $r$ is any other non-zero relation in the kernel of the map, which is not a multiple of $x_2f_1$ then by the same argument it coincides with $g_0$ modulo $x_2$ (up to multiplication with scalars). Then $g_0-r$ gives a relation divisible by $x_2$ which has to be a multiple of $x_2f_1$. Hence there is no further relation and the map is surjective with kernel spanned by $x_2f_1$ and $g_0$. 

The argument for the other weight spaces is analogous and gives
$g_0= y_1z_2-ay_2z_1+\tilde g_0$, $g_1=y_0z_1-\frac{1}{a}y_2z_2+\tilde g_1$, and 
 $g_2= y_0z_2-ay_1z_1+\tilde g_2$ for some $\tilde g_i$ divisible by $x_2$.

To determine $a$ consider in $\bar R^6_2$ the relation
\begin{equation}\label{eq: x0g0}
x_2g_0-y_1f_1+ay_2f_0  \equiv (a-1)y_1y_2^2+ay_0^2y_2-y_0y_1^2 \mod x_2^2.
\end{equation}
which does not contain $z_1, z_2$ modulo $x_2^2$ and thus is in
\[J_2^6 = \langle y_0f_2 = y_0^2y_2-y_0y_1^2+x_2^4y_0\rangle\]
modulo $x_2^2$. This is only possible if $a=1$.

It remains to show that the polynomials $\tilde g_i$ can be normalised as claimed. Substituting $a=1$ in \eqref{eq: x0g0} we see that the relation  $x_2g_0-y_0f_2-y_1f_1+y_2f_0$ in $\bar R^6_2$ 
is divisible by $x_2^2$. Since $I_1^4=\langle f_1\rangle$ there exists an $\alpha$ such that
\[x_2g_0-y_0f_2-y_1f_1+y_2f_0 = \alpha x_2^2f_1\]
Replacing $g_0$ by $g_0-\alpha x_2f_1$ gives 
\[ g_0 = y_1z_2-y_2z_1+x_2^3y_0.\]

The same argument works for $g_2$ looking at the relation 
\[x_2g_2-y_0f_1+y_1f_0-y_2f_2\]
and for $g_1$ looking at the relation
\[x_2g_1-y_0f_0+y_2f_1+h_0.\]
This concludes the proof.
\end{proof}
\begin{rem}
 One can easily check that with these choices we have the following syzygies of degree $6$:
 \begin{gather*}
x_2g_0-y_0f_2-y_1f_1+y_2f_0=0,\\
x_2g_1-y_0f_0+y_2f_1+h_0=0,\\
x_2g_2-y_0f_1+y_1f_0-y_2f_2=0. 
\end{gather*}
\end{rem}
\begin{lem}\label{lem: relations H}
  In $R^6$ the quadratic monomials in $z_1$, $z_2$ satisfy relations
\begin{gather*}
H_0=z_1z_2 -x_2^2y_0y_2+y_1F_2=0,\\
H_1=z_2^2 - x_2^2y_0^2+y_2F_2=0,\\
H_2= z_1^2 - x_2^2y_2^2+y_0F_2=0
\end{gather*}
\end{lem}
\begin{proof}
 Since $R^6 = S^6$ we know that we can express quadratic polynomials in $z_1$ and $z_2$ as elements in $S$. To compute one such expression we use that $x_2$ is a non-zero-divisor. Thus 
\begin{align*}
 x_2^2z_1^2 &\equiv (-y_0^2+y_1y_2)^2 &&\mod f_0\\
&= y_0(y_0^3-2y_0y_1y_2+y_2^3) - y_2^2(y_0y_2-y_1^2)\\
&=y_0h_0 -y_0x_2^2F_2-y_2^2f_2 +x_2^4y_2^2\\
&\equiv x_2^4y_2^2-x_2^2y_0F_2 && \mod f_2, h\\
&\implies z_1^2 - x_2^2y_2^2+y_0F_2\in  I_2^6.
\end{align*}
 We repeat the calculation for $z_2^2$,
\begin{align*}
 x_2^2z_2^2 &\equiv (-y_2^2+y_0y_1)^2 &&\mod f_1\\
&= y_2(y_0^3-2y_0y_1y_2+y_2^3) - y_0^2(y_0y_2-y_1^2)\\
&\equiv x_2^4y_0^2-x_2^2y_2F_2 && \mod f_2, h\\
&\implies z_2^2 - x_2^2y_0^2+y_2F_2\in  I_1^6,
\end{align*}
and for $z_1z_2$,
\begin{align*}
 x_2^2z_2z_2 & \equiv ( -y_0^2+y_1y_2)(-y_0y_1+y_2^2) &&\mod f_0, f_1\\
& =  y_0y_2(y_1^2-y_0y_2)+ y_1(y_0^3-2y_0y_1y_2 + y_2^3) \\
& \equiv x_2^4y_0y_2-x_2^2y_1F_2 && \mod f_2, h\\
&\implies z_1z_2 -x_2^2y_0y_2+y_1F_2 \in I_0^6
\end{align*}
\end{proof}

\begin{thm}\label{thm: restricted canonical ring Z/3 Godeaux}
Let $X$ be a Gorenstein stable Godeaux surface with torsion group  $T(X) \isom  \IZ/3$ and let $f\colon Y \to X$ be the corresponding triple cover. 
If $\tilde C =\inverse f C$ is the preimage of a paracanonical curve in $X$ and $\tilde M = K_Y|_{\tilde C}$ then 
\[ R(\tilde C, \tilde M ) = R = \IC[x_2, y_0, y_1, y_2, z_1, z_2]/(f_0, f_1, f_2, g_0, g_1, g_2, h_0, H_0, H_1, H_2),\]
with generators and relations chosen as above.

Therefore the canonical ring of $X$ can be described as in \cite{reid78} by lifting relations and syzygies to $\IC[x_1,x_2, y_0, y_1, y_2, z_1, z_2]$.
\end{thm}
A more conceptual approach to this ring is explained in \cite{reid15}.
\begin{proof}
By the above results there is a surjection from the ring on the right hand side onto $R$ which is a bijection in even degrees and in degree up to $5$. 

Assume that some polynomial  $r(x_2, \dots, z_2)$ of odd degree at least $7$ is zero in $R$. By Lemma \ref{lem: S and z's} it cannot be contained in $S$ and thus has to involve $z_1$ or $z_2$. Using the relations in $I$ we see that 
\[ r \equiv z_1r_1(y_0, y_1, y_2) + r_2(x_2, y_0, y_1, y_2)\mod I.\]
Reducing modulo $z_1$ we see that $r_2\in J$, so actually $r\equiv r \equiv z_1r_1(y_0, y_1, y_2) \mod I$. However, $z_1$ is not a zero-divisor, thus $r_1\in I$ and consequently $r\in I$ as claimed.

\end{proof}

\subsection{Remarks on the case $|T(X)|=2$} 
The description of the universal cover of 
Godeaux surfaces with torsion $\IZ/2$ has been treated in \cite{catanese-debarre89} and from a slightly different point of view in \cite{coughlan16}. We did not attempt to extend  this description to Gorenstein stable Godeaux surfaces but believe it should go through: In the smooth case, the following is the  starting point of both constructions.

\begin{lem}
 Let $f\colon Y\to X$ be the natural double cover of a Gorenstein stable Godeaux surface with torsion $\IZ/2$. Then a canonical curve $D$ of $Y$ is honestly hyperelliptic of arithmetic genus $3$.
\end{lem}
\begin{proof}
 The canonical curve is a double cover of a paracanonical curve $C\subset X$, which is irreducible of arithmetic genus $2$ by Lemma \ref{lem: paracanonical 1}. Thus the canonical pencil of $C$ defines a polynomial subring of $R(D, K_D)$ and since the image of the multiplication map $H^0(K_D)\tensor H^0(K_D) \to H^0(2K_D)$ is contained in $H^0(2K_C)$ we are done. 
\end{proof}

\section{Proof of Theorem A}
In this short section we quickly deduce Theorem A from the descriptions of the canonical rings in the previous section.

Let $X$ be a Gorenstein stable Godeaux surface with torsion of order  $3\leq d\leq  5$ and $f\colon Y \to X$ the cover associated to $T(X)\isom \IZ/d$. Then we have seen in the preceding section that the canonical ring of $Y$, including the action of $\IZ/d$, is uniquely determined up to the choice of some parameters parametrised by an open set in some projective space. We know from \cite{reid78} that the parameter space is non-empty in each case and thus the image of each of these families is a uni-rational irreducible component. 

Since the general element corresponds to  a smooth Godeaux surface and smoothness is an open condition every surface in the family is smoothable.

\section{The canonical ring of a simply connected stable Godeaux surface}

The canonical rings of simply-connected (stable) Godeaux surfaces remain elusive (see however \cite{catanese-pignatelli00}). When we consider degenerations, the constructions become much more explicit and thus actual computations are possible. In this section we give a description of the canonical ring of a simply connected Gorenstein stable Godeaux surface which however rather serves as an indication of the complexity of the problem than as a starting point for a general structure theory. 
The computation was carried out by the second author in discussion with Roberto Pignatelli.

Let $a,b,c$ be homogeneous coordinates in the plane $\bar X = \IP^2$ and consider the plane quartic $\bar D$ which is union of 
\begin{gather*}
 C = \{f = a^2-6ab+b^2-c^2=0\},\\
 L = \{a=0\} =\Proj( \IC[b,c]),\\
 L' = \{b=0\} =  \Proj( \IC[a,c]).
\end{gather*}
We define an involution $\tau$ on the normalisation $\bar D^\nu$ of $\bar D$ that preserves the conic and interchanges the lines by its action on functions: 
\begin{align*}
 \text{on $C$} && \tau^*(a,b,c) = (-a, -b, c), \\
 L\isom L' && \tau^*(b,c) = -\frac12\left(a+c, 3a-c\right).
\end{align*}
Then as explained in \cite[Section~3.B, Case $(P_2)$; Prop.~3.16]{FPR16b} the triple $(\bar X, \bar D, \tau)$ gives rise to a non-normal Gorenstein stable Godeaux surface $X$, by glueing $\bar D$ to itself as prescribed by $\tau$. 

Let  $\pi\colon  \bar X \to X$ be the normalisation of $X$, and $D\subset X$ the non-normal locus. Then $\inverse \pi (D) = \bar D$  and $D$ consists of two irreducible components: a rational curve with one node, which is the image of the conic, and a rational curve with a triple point, which is the image of the lines. These two components meet transversally at the singular points, that is, $D$ has a unique singular point which  (analytically locally) looks like the coordinate axes in $\IC^5$.

Our computation, which we will only sketch, is based on the following result of Koll\'ar, which we state  in a simplified version.
\begin{prop}[{\cite[Prop.~5.8]{KollarSMMP}}]\label{prop: sections}
Let $X$ be a Gorenstein stable  surface. Define the different  $\Delta = \Diff_{\bar D^\nu}(0)$ by the equality $(K_{\bar X}+\bar D)\restr{ \bar D}  = K_{\bar D}+ \Delta$. 

Then a section $s\in H^0(\bar X, m(K_{\bar X}+\bar D))$ descends to a section in $H^0(X, mK_X)$ if and only if the image of $s$ in $H^0(\bar D^\nu, m( K_{\bar D}+ \Delta))$ under the Residue map is $\tau$-invariant if $m$ is even respectively $\tau$-anti-invariant if $m$ is odd.
 \end{prop}
To do the actual computation we need to pick explicit generators of the respective bundles and compute the residue maps. We choose as  generator of $\omega_{\bar X}(\bar D)$ the rational 2-form
\[ \bar\omega = \frac{abc}{abf}\left(\frac{\mathrm{d}a\wedge \mathrm{d}b}{ab}+\frac{\mathrm{d}b\wedge \mathrm{d}c}{bc}+\frac{\mathrm{d}c\wedge \mathrm{d}a}{ca}\right) =  \frac{c}{f}\left(\frac{\mathrm{d}a\wedge \mathrm{d}b}{ab}+\frac{\mathrm{d}b\wedge \mathrm{d}c}{bc}+\frac{\mathrm{d}c\wedge \mathrm{d}a}{ca}\right)\]
which gives an isomorphism $R(\bar X, K_{\bar X}+\bar D)\isom \IC[a,b,c]$. 
As generators for $\omega_{L}(\bar D-L)$, $\omega_{L'}(\bar D-L')$, respectively $\omega_{C}(L+L')$ we choose the forms
\begin{gather*} \omega = \Res_{L} (\bar \omega) = \frac c {f} \left(\frac {\mathrm{d}b}b -\frac{\mathrm{d}c}c\right)\\
\omega' = \Res_{L'} (\bar\omega)  = \frac c{f} \left( - \frac{\mathrm{d}a}a+\frac{\mathrm{d}c}c\right)\\
\eta = \Res_{C}(\bar \omega) 
\end{gather*}
Note that  $\tau^*\omega = \omega'$, that is, with the choices made,  the residue maps are compatible with the identifications $R(L, \omega_{L}(C+L'))\isom \IC[b,c]$ and $ R(L', \omega_{L'}(C+L))\isom \IC[a,c]$ and we have by Proposition \ref{prop: sections}
\[ R(X, K_X) \isom \left\{g \in \IC[a,b,c]\left| \begin{matrix} \text{$g$ is contained in $\IC[a,b,c^2]$ modulo $f$}\\  g(a,0,c) = g(0, 1/2(a+c), 1/2(3a+c))\end{matrix}\right.\right\},\]
where the apparent change of signs in the involutions it due to the fact that we take anti-invariant sections in odd degrees and invariant sections in even degrees. The generators of the resulting ring can be computed with a computer algebra system and we get:
\begin{prop}
 The canonical ring of $X$ is generated as a subring of $\IC[a,b,c]$ by
% {\small \begin{align*}
% & ab, 3a^2+3b^2+c^2,\\
% & ab^2, a^2b,\\
% & 12b^3-a^2c+6abc-b^2c+8ac^2-4bc^2+c^3\\
% & 12a^3-a^2c+6abc-b^2c-4ac^2+8bc^2+c^3\\
% & abc^2, ab^3,\\
% & 9a^3c-45a^2bc-45ab^2c+9b^3c+15a^2c^2+15b^2c^2-9ac^3-9bc^3+c^4\\
% & 27b^4+9a^2bc-54ab^2c+9b^3c+30a^2c^2+3b^2c^2-9bc^3+2c^4\\
% & ab^2c^2\\
% & 18a^2b^2c-108ab^3c+18b^4c+39a^3c^2+33b^3c^2-a^2c^3+6abc^3-19b^2c^3-7ac^4-bc^4+c^5\\
% & 3a^4c-612ab^3c+105b^4c+233a^3c^2+199b^3c^2-9a^2c^3+18abc^3-111b^2c^3-41ac^4-7bc^4+6c^5
%  \end{align*}}
 { \begin{align*}
& ab, 3a^2+3b^2+c^2,\\
& ab^2, a^2b,\\
& 12b^3-a^2c+6abc-b^2c+8ac^2-4bc^2+c^3,\\
& 12a^3-a^2c+6abc-b^2c-4ac^2+8bc^2+c^3,\\
& abc^2, ab^3,\\
& 9a^3c-45a^2bc-45ab^2c+9b^3c+15a^2c^2+15b^2c^2-9ac^3-9bc^3+c^4,\\
& 27b^4+9a^2bc-54ab^2c+9b^3c+30a^2c^2+3b^2c^2-9bc^3+2c^4,\\
& ab^2c^2,\\
& 18a^2b^2c-108ab^3c+18b^4c+39a^3c^2+33b^3c^2-a^2c^3
\\& \qquad+6abc^3-19b^2c^3-7ac^4-bc^4+c^5,\\
& 3a^4c-612ab^3c+105b^4c+233a^3c^2+199b^3c^2-9a^2c^3
\\& \qquad+18abc^3-111b^2c^3-41ac^4-7bc^4+6c^5,
 \end{align*}}
and there are $54$ relation in degrees $(6^6, 7^{12},  8^{18},9^{12}, 10^6)$.
In particular, $X$ embeds canonically as a codimension $10$ subvariety in $\IP(2^2, 3^4, 4^4, 5^3)$.
 \end{prop}
The numbers and degrees of generators and relations in this ring are the same as in the other example that was computed in \cite{rollenske16}, which makes hope that a general structure theory is lurking in the background.

Note that by 
\cite[Example~47]{liu-rollenske14} the unique degenerate cusp of $X$ is a base point of the  $2$-canonical map, and neither the 3-canonical nor the 4-canonical map embed the non-normal locus of $X$. 

%   \bibliographystyle{../halpha}
%  \bibliography{../srollens}
%  
 \def\cprime{$'$}

\end{document}